\documentclass[journal ]{new-aiaa}
\usepackage[utf8]{inputenc}
\usepackage{textcomp}

\usepackage{graphicx}
\usepackage{amsmath, amsthm}
\usepackage[version=4]{mhchem}
\usepackage{siunitx}
\usepackage{longtable,tabularx}
\usepackage{mathtools}
\usepackage{caption}
\usepackage{subcaption}
\usepackage{todonotes}
\newtheorem{theorem}{Theorem}
\newtheorem{lemma}[theorem]{Lemma}

\theoremstyle{definition}

\title{Equivalence of Dubins Path on Sphere with Geographic Coordinates and Moving Frames}

\author{Deepak Prakash Kumar \footnote{Graduate Student, Mechanical Engineering, 3123 TAMU, College Station, TX 77843.}, Swaroop Darbha\footnote{Professor, Mechanical Engineering, 3123 TAMU, College Station, TX 77843.}, Satyanarayana G. Manyam\footnote{Research Scientist, DCS Corporation, 4023 Col. Glenn Hwy, Dayton, OH.}, David W. Casbeer\footnote{Sr. Engineer, Control Science Center, AFRL, AIAA Associate Fellow}, Meir Pachter\footnote{Professor, Department of Electrical and Computer Engineering, Air Force Institute of Technology, and AIAA Associate Fellow\\ \textnormal{\footnotesize{DISTRIBUTION STATEMENT A. Approved for public release. Distribution is unlimited. AFRL-2024-2304; Cleared 04/25/2024}}}}

\begin{document}

\maketitle




\section{Introduction}

Path planning for curvature-constrained vehicles has been a growing area of interest in the literature due to the increasing civilian and military applications of such vehicles, which include fixed-wing unmanned aerial vehicles. Curvature-constrained vehicles broadly include vehicles that need to maintain a minimum non-zero longitudinal speed and have a bound on the rate of change of the orientation/heading. A popular model used to model such a vehicle is by Dubins \cite{Dubins}, wherein the author modeled and solved the classical Markov-Dubins problem to navigate the vehicle from a given location and heading to another on a plane using a least distance path. In \cite{Dubins}, the author showed that the optimal path for this problem is of type $CSC, CCC,$ or a degenerate path of the same, where $C \in \{L, R\}$ denotes a left or right turn with the minimum turning radius, and $S$ denotes a straight line segment. A variant of the problem, wherein the vehicle can move forward or back with a bounded speed, was studied in \cite{Reeds_Shepp}. Such a vehicle is called a Reeds-Shepp vehicle.

The results in \cite{Dubins, Reeds_Shepp} were obtained without utilizing Pontryagin's Maximum Principle (PMP) \cite{PMP}, a first-order necessary condition for optimal control problems. Studies such as \cite{sussman_geometric_examples, boissonat} utilized PMP for the considered problems and obtained the results through simpler proofs. Recently, PMP along with phase portraits have been utilized for the classical Markov-Dubins problem and a variant of the problem, known as the weighted Markov-Dubins problem, in \cite{phase_portrait_kaya} and \cite{weighted_Markov_Dubins} to obtain the optimal path through a set of identified cases. 

Many variants of the classical Markov-Dubins problem, wherein the vehicle is considered to move on a plane, have been studied extensively \cite{sinistral/dextral, dubins_circle, weighted_Markov_Dubins}. However, fewer studies have considered path planning in 3D to travel from one configuration (location and orientation) to another. In \cite{monroy}, path planning for a Dubins vehicle on surfaces such as a sphere and hyperbolic surface was considered. In particular, it was shown that the planar Dubins results extend to a unit sphere, wherein the vehicle's minimum turning radius is equal to $r = \frac{1}{\sqrt{2}}$. In particular, the optimal path was shown to be of type $CGC, CCC,$ or a degenerate path of the same, where $C = L, R$ denotes a tight left or right turn of radius $r$, and $G$ denotes a great circular arc. Time-optimal control for a generic system on $SO (3)$, which is the group of special orthogonal matrices including rotation matrices, has been studied in \cite{time_optimal_synthesis_SO3} and \cite{time_optimal_control_satellite} for one and two control input systems, respectively. However, to apply the results of the study for the Dubins model on a sphere considered in \cite{monroy, 3D_Dubins_sphere}, the radius of turn of the vehicle reduces to $r = \frac{1}{\sqrt{2}}$, which is the same as in \cite{monroy}. Furthermore, the number of concatenations shown in \cite{time_optimal_synthesis_SO3} cannot be applied to the Dubins model since the parameter of the model in \cite{time_optimal_synthesis_SO3} equals $\frac{\pi}{4}$ for the Dubins model, for which the result in \cite{time_optimal_synthesis_SO3} is not provided. 

In \cite{3D_Dubins_sphere}, the path planning problem for a Dubins vehicle on a sphere was addressed by modeling the vehicle using a Sabban frame. The authors showed that the Dubins result extends to the sphere for $r \leq \frac{1}{2}.$ In our earlier work \cite{free_terminal_sphere}, we studied the motion planning problem using the same model for a problem with free terminal heading angle,  and showed that the optimal path is of type $CC, CG,$ or a degenerate path of the same for $r \leq \frac{1}{2}$.

The previously surveyed studies for motion planning on a sphere or $SO(3)$ employ moving frames, wherein the chosen frame denotes the vehicle's configuration, i.e., the location and orientation. In practice, motion planning for a vehicle on a sphere, such as the earth\footnotemark\, (approximately), can also be approached by parameterizing the vehicle's configuration using spherical coordinates along with a heading angle, as shown in Fig.~\ref{fig: frames_sphere}. The model is used, particularly in the aerospace community, due to the ease of interpreting the motion of the vehicle. However, to the best of our knowledge, motion planning for a curvature-constrained vehicle moving on a sphere through this modeling approach has not been formulated. Furthermore, the equivalence of modeling a Dubins vehicle through a moving frame approach to the model employing spherical coordinates has not been explored in the literature.

\footnotetext{Though the earth is considered approximately a sphere, we do not account for the difference in the radius with the location on the earth.}

\begin{figure}[htb!]
    \centering
    \includegraphics[width = 0.25\linewidth]{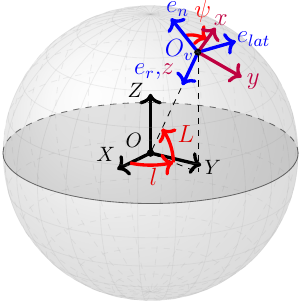}
    \caption{Frames considered on a sphere and lateral force on vehicle}
    \label{fig: frames_sphere}
\end{figure}

Hence, the main contributions of this article include
\begin{enumerate}
    \item Deriving the motion planning problem on a sphere using spherical coordinates,
    \item Revising the motion planning results through the moving frame approach and showing equivalence of the two considered models, thereby solving the motion planning problem for a model using spherical coordinates.
\end{enumerate}

\section{Problem Formulation and Hamiltonian Construction}

The problem of determining the path of shortest length on a unit sphere connecting two configurations for a geodesic-curvature constrained Dubins was formulated in \cite{3D_Dubins_sphere} as the following variational problem:
\begin{equation}
    J = \min \int_0^L 1 ds
\end{equation}
subject to
\begin{equation} \label{eq: Sabban_frame_equations}
    \frac{d{\mathbf X}}{ds} = {\mathbf T}(s), \quad \quad \frac{d {\mathbf T}}{ds} = - {\mathbf X}(s) + u_g(s) {\mathbf N}(s), \quad \quad \frac{d {\mathbf N}}{ds} = - u_g(s){\mathbf T}(s), 
\end{equation}
and the boundary conditions ${\mathcal R}(0) = I_3, {\mathcal R}(L) = R_f$.
Here, $u_g \in [-U_{max}, U_{max}],$ where $U_{max}$ is a parameter depending on the chosen vehicle. The authors set up the Hamiltonian using adjoint vectors $\lambda_1, \lambda_2, \lambda_3$ corresponding to the three constraints. However, in this article, the Hamiltonian setup utilized by Monroy \cite{monroy} will be utilized to easily show non-triviality condition. 
Similar to \cite{monroy}, the equations in Eq.~\eqref{eq: Sabban_frame_equations} can be assembled as
\begin{align} \label{eq: Lie_group_evolution}
    \frac{dg}{ds} = \overrightarrow{l}_1 \left(g (s) \right) - u_g (s) \overrightarrow{L}_{12} \left(g (s) \right),
\end{align}
where $g$ is an element of the Lie group $SO (3)$ and 
$g (s) = \begin{pmatrix}
    \textbf{X} (s) & \textbf{T} (s) & \textbf{N} (s)
\end{pmatrix}.$
Further, $s \rightarrow g (s)$ is a curve on the Lie group, and $\overrightarrow{l}_1$ and $\overrightarrow{L}_{12}$ are left-invariant vector fields on the Lie group,
whose value at the identity of the Lie group is given by \cite{monroy}
\begin{align*}
    l_1 = \begin{pmatrix}
        0 & -1 & 0 \\
        1 & 0 & 0 \\
        0 & 0 & 0
    \end{pmatrix}, \quad L_{12} = \begin{pmatrix}
        0 & 0 & 0 \\
        0 & 0 & 1 \\
        0 & -1 & 0
    \end{pmatrix}.
\end{align*}


\textbf{Remark:} The left-invariant vector field $\overrightarrow{l}_1 \left(g (t) \right)$ and $\overrightarrow{L}_{12} \left(g (t) \right)$ can be expressed in terms of its value at the identity as
\begin{align*}
    \overrightarrow{l}_1 \left(g (s) \right) &= g (s) l_1 = \begin{pmatrix}
        \mathbf{X} (s) & \textbf{T} (s) & \textbf{N} (s)
    \end{pmatrix} \begin{pmatrix}
        0 & -1 & 0 \\
        1 & 0 & 0 \\
        0 & 0 & 0
    \end{pmatrix} = \begin{pmatrix}
        \textbf{T} (s) & -\textbf{X} (s) & \textbf{0}
    \end{pmatrix}, \\
    \overrightarrow{L}_{12} \left(g (s) \right) &= g (s) L_{12} = \begin{pmatrix}
        \mathbf{X} (s) & \textbf{T} (s) & \textbf{N} (s)
    \end{pmatrix} \begin{pmatrix}
        0 & 0 & 0 \\
        0 & 0 & 1 \\
        0 & -1 & 0
    \end{pmatrix} = \begin{pmatrix}
        \textbf{0} & -\textbf{N} (s) & \textbf{T} (s)
    \end{pmatrix}.
\end{align*}
Hence, the system evolves based on the two left-invariant vector fields $\overrightarrow{l}_1 \left(g (t) \right)$ and $\overrightarrow{L}_{12} \left(g (t) \right)$, since $\overrightarrow{l}_1 \left(g (t) \right) - u_g (s) \overrightarrow{L}_{12} \left(g (t) \right)$ yields the differential equation for $\frac{d \mathbf{X} (s)}{ds}, \frac{d \mathbf{T} (s)}{ds},$ and $\frac{d \textbf{N} (s)}{ds}$ obtained from the Sabban frame.

Similar to Monroy \cite{monroy}, Pontryagin's Maximum Principle is applied for the symplectic formalism \cite{Jurdjevic}. The Hamiltonian is obtained as $H (\zeta) = \zeta_0 + h_1 (\zeta) - u_g H_{12} (\zeta),$
where $\zeta \in T^* (G)$ is a one-form on the cotangent bundle of the Lie group $G$. Here, $h_1$ and $H_{12}$ are functions on $T^* G,$ which is the cotangent bundle of $G$. In the above formulation, $\zeta_0$ does not depend on $u_g$ and can be taken as an arbitrary parameter. Hence, it can be normalized to $0$ or $-1$ \cite{monroy}. 


From Pontryagin's Maximum Principle \cite{monroy, Jurdjevic}, the optimal curvature $\kappa (s)$ and the Hamiltonian are immediately obtained when $H_{12} \neq 0$ as 
\begin{align} 
    \kappa (s) &:= -U_{max} sign(H_{12} (\zeta (s))), \label{eq: optimal_curvature} \\
    H (\zeta (s), \kappa, \lambda) &= -\lambda + h_1 (\zeta (s)) - \kappa (s) H_{12} (\zeta (s)), \label{eq: Hamiltonian_rewritten}
\end{align}
where $\lambda = \{0, 1 \}$. It should be noted that the same Hamiltonian expression is obtained as given in \cite{monroy} except for a change in the sign of the optimal curvature. Hence, from \cite{monroy}, the evolution of the functions $h_1, h_2,$ and $H_{12}$ are given by
\begin{equation*}
    \dot{h}_1 (\zeta (s)) = -\kappa (s) h_2 (\zeta (s)), \quad
    \dot{h}_2 (\zeta (s)) = H_{12} (\zeta (s)) + \kappa (s) h_1 (\zeta (s)), \quad
    \dot{H}_{12} (\zeta (s)) = -h_2 (\zeta (s)),
\end{equation*}
where $h_2$ is the function corresponding to the Hamiltonian vector field $\overrightarrow{l}_2$, whose value at the identity of the Lie group is given by $l_2 = [L_{12}, l_1],$
since $l_1, l_2, L_{12}$ forms a basis for the Lie algebra of the Lie group $G$.

\textbf{Remark:} In this model, $-H_{12}, h_2, h_1$ play the role of $A, B,$ and $C$ in \cite{3D_Dubins_sphere}, respectively.

When $H_{12} \equiv 0,$ the corresponding optimal control action is established in the following lemma. 
\begin{lemma}
    If $H_{12} \equiv 0,$ then $\lambda$ cannot be zero; further, for $\lambda = 1,$ $\kappa \equiv 0.$
\end{lemma}
\begin{proof}
    The proof follows the same steps utilized in Monroy \cite{monroy} in Lemma~5.1 and Lemma~5.2.
\end{proof}

From this lemma, it follows that the optimal control actions are as follows:
\begin{equation} \label{eq: optimal_control_inputs}
    \kappa (s) \equiv
    \begin{cases}
        - U_{max}, & H_{12} (\zeta (s)) > 0, \lambda \in \{0, 1 \} \\
        U_{max}, & H_{12} (\zeta (s)) < 0, \lambda \in \{0, 1\} \\
        0, & H_{12} (\zeta (s)) \equiv 0, \lambda = 1.
    \end{cases}
\end{equation}
Henceforth, $H_{12} (\zeta (s))$ will be denoted as $H_{12} (s)$ for brevity.

\section{Characterization of Optimal Paths for Model with Sabban Frame}

Using the obtained control actions, it follows that the optimal path is a concatenation of segments corresponding to $\kappa = -U_{max}, 0,$ and $U_{max}.$ It should be noted that $\kappa (s) \equiv 0$ corresponds to an arc of a great circle, and $\kappa (s) = \pm U_{max}$ correspond to an arc of a small circular arc of radius $r = \frac{1}{\sqrt{1 + U_{max}^2}}$ \cite{3D_Dubins_sphere}. 
The three identified segments are depicted in Fig.~\ref{fig: turns_sphere}.

\begin{figure}[htb!]
    \centering
    \includegraphics[width = 0.3\linewidth]{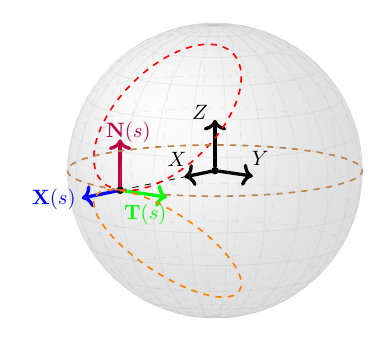}
    \caption{Turns on a sphere}
    \label{fig: turns_sphere}
\end{figure}

From the identified control inputs, it can be observed that the control input is dependent on the scalar function $H_{12} (s).$ 
Hence, it is desired to obtain an equation for the closed-form evolution of $H_{12} (s).$
Noting that $-h_2 (s) = \frac{d H_{12} (s)}{ds}$ and $\frac{d h_2 (s)}{d s} = H_{12} (s) + \kappa (s) h_1 (s),$ and $\frac{d H_{12} (s)}{ds}$ is differentiable, $\frac{d^2 H_{12} (s)}{ds^2} = -H_{12} (s) - \kappa (s) h_1 (s).$
Since $H \equiv 0$ for the problem formulation considered, from Eq.~\eqref{eq: Hamiltonian_rewritten},
$h_1 (s) = \lambda + \kappa (s) H_{12} (s).$
Therefore,
the evolution of $H_{12} (s)$ can be obtained by solving
\begin{equation} \label{eq: evolution_H12}
    \frac{d^2 H_{12} (s)}{d s^2} + \left(1 + \kappa^2 (s) \right) H_{12} (s) = -\lambda \kappa (s),
\end{equation}
a second-order ODE, 
where $\lambda \in \{0, 1\}$. It should be noted here that $\frac{d^2 H_{12} (s)}{d s^2}$ is piecewise continuous since $-H_{12} (s) - \kappa (s) h_1 (s)$ is piecewise continuous. Therefore, the solution for $H_{12} (s)$ and $\frac{d H_{12} (s)}{ds}$ can be obtained for each interval over which $\kappa (s)$ is continuous. 


\section{Equivalence to Model with Geographic Coordinates}

The optimal path problem on the sphere can alternately be posed using spherical coordinates $L$, the latitude, and $l$, the longitude, to represent the vehicle's location. Consider a vehicle traveling over the surface of a sphere of radius $R_e$
at a constant altitude $h$ and at a constant longitudinal speed $v$. 
In this model, a lateral force $F$ is considered a control input that controls the rate of change of the heading angle of the vehicle. 
To provide a mathematical model of the vehicle, three frames are considered, shown in Fig.~\ref{fig: frames_sphere}:
\begin{itemize}
    \item An inertial frame $I$ attached to the center of the sphere $O$ with axes $X, Y, Z$, where $Z$ points towards the north pole.
    \item A body frame $B$ attached to the vehicle's location $O_v$ with axes $e_n, e_l, e_r$ pointed towards the north pole, along the latitude, and radially inwards, respectively.
    \item A navigation frame $N$ attached to $O_v$ with axes $x, y, z$ such that $z = e_r,$ and $x$ and $y$ are in the same plane as $e_n$ and $e_l$. Here, $x$ is pointed along the vehicle's longitudinal direction. Further, $y$ is along the lateral direction of the vehicle, along which $F$ acts. Furthermore, the angle made by $x$ with respect to $e_n$ denotes the heading angle $\psi$. 
\end{itemize}


It is desired to derive the equations of the vehicle's motion in the inertial frame $I$. Applying Newton's laws of motion in the inertial frame $I$ expressed in the body frame $B$,
\begin{equation} \label{eq: equation_of_motion_vehicle}
    \prescript{}{B}{\mathbf{F}}^N = m \left(\frac{d}{dt} \prescript{}{B}{\mathbf{v}}^{N} + \prescript{}{B}{\boldsymbol{\omega}}^B \prescript{}{B}{\mathbf{v}}^{N} \right),
\end{equation}
where $\prescript{}{B}{\mathbf{F}}^N$ is the force acting on the body expressed in the body frame $B$, $\prescript{}{B}{\mathbf{v}}^{N}$ is the velocity of the vehicle expressed in the body frame, and $\prescript{}{B}{\boldsymbol{\omega}}^B$ is the angular velocity matrix of the body frame as seen by an observer in the ground frame, whose components are expressed in the body frame.

The force acting on the vehicle in the navigation frame is along the $y$-axis, and the vehicle's velocity in the navigation frame is along the $x$-axis, with a magnitude equal to $v$. Hence, in the body frame, these vectors are obtained as $\prescript{}{B}{\mathbf{F}}^N = \begin{pmatrix}
    -F_y \sin{\psi} & F_y \cos{\psi} & 0
\end{pmatrix}^T, \prescript{}{B}{\mathbf{v}}^{N} = \begin{pmatrix}
        v \cos{\psi} &
        v \sin{\psi} &
        0
    \end{pmatrix}^T$.
The angular velocity matrix $\prescript{}{B}{\boldsymbol{\omega}}^B$ can be obtained using the net rotation matrix $R_{net}$ relating the body frame and the inertial frame as
$\prescript{}{B}{\boldsymbol{\omega}}^B = R_{net}^T (t) \frac{d R_{net} (t)}{d t},$
where
\begin{equation*}
    R_{net} (t) = R_z (l) R_y \left(-L - \frac{\pi}{2} \right),
\end{equation*}
to align the $X, Y, Z$ axes with $e_n, e_l,$ and $e_r$ axis, respectively. 
Therefore, $\prescript{}{B}{\boldsymbol{\omega}}^B$ is given by
\begin{align*}
    \prescript{}{B}{\boldsymbol{\omega}}^B &= \begin{pmatrix}
        0 & \sin{\left(L (t) \right)} \dot{l} (t) & -\dot{L} (t) \\
        -\sin{\left(L (t) \right)} \dot{l} (t) & 0 & -\cos{\left(L (t) \right)} \dot{l} (t) \\
        \dot{L} (t) & \cos{\left(L (t) \right)} \dot{l} (t) & 0
    \end{pmatrix}.
\end{align*}

Substituting the obtained expressions for $\prescript{}{B}{\mathbf{F}}^N, \prescript{}{B}{\mathbf{v}}^N,$ and $\prescript{}{B}{\boldsymbol{\omega}}^B$ in Eq.~\eqref{eq: equation_of_motion_vehicle}, the equation obtained relating components of $F_y$ to the linear and angular speeds is
\begin{align*}
    \frac{1}{m} \begin{pmatrix}
        -F_y \sin{\psi} \\
        F_y \cos{\psi} \\
        0
    \end{pmatrix} &= v \begin{pmatrix}
        -\sin{\psi} \\
        \cos{\psi} \\
        0
    \end{pmatrix} + \begin{pmatrix}
        0 & \sin{\left(L (t) \right)} \dot{l} (t) & -\dot{L} (t) \\
        -\sin{\left(L (t) \right)} \dot{l} (t) & 0 & -\cos{\left(L (t) \right)} \dot{l} (t) \\
        \dot{L} (t) & \cos{\left(L (t) \right)} \dot{l} (t) & 0
    \end{pmatrix} \begin{pmatrix}
        v \cos{\psi} \\
        v \sin{\psi} \\
        0
    \end{pmatrix}.
\end{align*}
Pre-multiplying the obtained equation with the vector $\begin{pmatrix}
    -\sin{\psi} & \cos{\psi} & 0
\end{pmatrix}$, and simplifying,
\begin{align} \label{eq: evolution_psi_dot}
    \frac{F_y}{m} &= v \dot{\psi} - \sin{L} \dot{l}.
\end{align}
However, from Fig.~\ref{fig: frames_sphere}, the evolution of $l$ and $L$ are given by $\frac{d L}{dt} = v \cos{\psi} \frac{1}{R_e + h}, \frac{d l}{d t} = v \sin{\psi} \frac{1}{\left(R_e + h \right) \cos{L}}.$
Here, in the expression for $\frac{d l}{d t}$, $\left(R_e + h \right) \cos{L}$ is the radius of the circle that the vehicle would traverse with speed $v \sin{\psi},$ 
as seen in Fig.~\ref{fig: evolution_l_diagram}.
\begin{figure}[htb!]
    \centering
    \includegraphics[width = 0.3\linewidth]{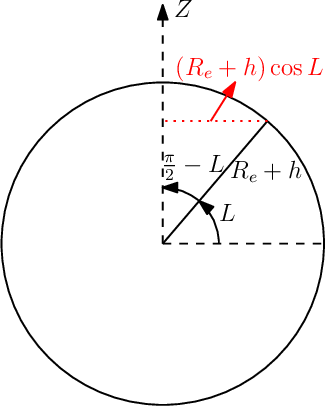}
    \caption{ Representation of radius of turn for evolution of $l$}
    \label{fig: evolution_l_diagram}
\end{figure}
Therefore, substituting the expression of $\dot{l}$ in Eq.~\eqref{eq: evolution_psi_dot} and rearranging, the evolution of $\psi$ can be obtained as $\frac{d \psi}{d t} = \frac{1}{R_e + h} \tan{L} \sin{\psi} + \frac{1}{m v} F_y.$

Non-dimensionalizing the above system, wherein $R_e + h$ is set to one to represent the vehicle traveling on a unit sphere, and $v$ is set to one, the differential equations can be rewritten as
\begin{align}
    \frac{d L}{dt} &= \cos{\psi}, \quad
    \frac{d l}{d t} = \frac{\sin{\psi}}{\cos{L}}, \quad
    \frac{d \psi}{d t} = \tan{L} \sin{\psi} + \frac{1}{\eta} u, \label{eq: Sabban_frame_equations_rate_change_heading}
\end{align}
where $-1 \leq u \leq 1$ is the control input, and $\eta$ is a parameter that dictates the tightest turning radius of the vehicle.
Therefore, the optimal control problem is formulated as
\begin{align}
    J = \min \int_0^{L_{tot}} 1 ds
\end{align}
subject to the constraints given in Eq.~\eqref{eq: Sabban_frame_equations_rate_change_heading}
and the boundary conditions $L (0) = 0, l (0) = 0, \psi (0) = 0, L (L_{tot}) = L_f, l (L_{tot}) = l_f, \psi (L_{tot}) = \psi_f.$

The optimal control actions for the described system are desired to be obtained using PMP, for which the Hamiltonian is defined using adjoint variables $e, \lambda_L (s), \lambda_l (s),$ and $\lambda_{\psi} (s)$ as
\begin{align} \label{eq: expression_Hamiltonian}
\begin{split}
    H &= e + \lambda_L \cos{\psi} + \lambda_l \frac{\sin{\psi}}{\cos{L}} + \lambda_\psi \tan{L} \sin{\psi} + \lambda_\psi \frac{1}{\eta} u,
\end{split}
\end{align}
where the dependence of the Hamiltonian on $s$, the adjoint variables, and the control inputs are not shown for brevity. From PMP, $e \leq 0$ and is a constant. Hence, $e$ can be scaled to be equal to $-1$ or $0.$ The adjoint equations are obtained as
\begin{align}
    \dot{\lambda}_L &= -\lambda_l \frac{\sin{\psi}}{\cos^2{L}} \sin{L} - \lambda_\psi \frac{\sin{\psi}}{\cos^2{L}}, \quad
    \dot{\lambda}_l = 0, \quad
    \dot{\lambda}_\psi = \lambda_L \sin{\psi} - \lambda_l \frac{\cos{\psi}}{\cos{L}} - \lambda_\psi \tan{L} \cos{\psi}. \label{eq: evolution_lambda_psi_dot}
\end{align}
The optimal control actions in these two cases are obtained in the following lemma.

\begin{lemma}
    The optimal control input for $e = -1$ is given by
    \begin{align*}
        u = \begin{cases}
            1 & \lambda_\psi > 0, \,\, e \in \{-1, 0\} \\
            -1 & \lambda_\psi < 0, \,\, e \in \{-1, 0\} \\
            0 & \lambda_\psi \equiv 0, \,\, e = -1
        \end{cases}.
    \end{align*}
    Further, if $\lambda_\psi \equiv 0$ and $e = 0,$ then no optimal control action exists. 
\end{lemma}
\begin{proof}
    Since the control input pointwise maximizes the Hamiltonian, and the Hamiltonian is linear in $u$, the optimal control action for $\lambda_\psi > 0$ and $\lambda_\psi < 0$ are immediate. Now, suppose $\lambda_\psi \equiv 0.$ Then, from Eq.~\eqref{eq: evolution_lambda_psi_dot}, it follows that
    \begin{align} \label{eq: lambda_psi_id_zero_equation}
        \lambda_\psi \equiv 0 \implies \lambda_L \sin{\psi} - \lambda_l \frac{\cos{\psi}}{\cos{L}} \equiv 0.
    \end{align}
    Furthermore, since the Hamiltonian in Eq.~\eqref{eq: expression_Hamiltonian} is identically zero and $\lambda_\psi \equiv 0$ is considered, it follows that
    \begin{align} \label{eq: lambda_psi_id_zero_equation_H_zero}
        \lambda_L \cos{\psi} + \lambda_l \frac{\sin{\psi}}{\cos{L}} = -e.
    \end{align}
    Using the above two equations, the solution obtained when $L \neq \pm \frac{\pi}{2}$ is $\lambda_L = -e \cos{\psi}, \lambda_l = -e \cos{L} \sin{\psi}$.
    If $e = 0,$ non-triviality condition is not satisfied since all adjoint variables are zero.
    
    Let $e = -1.$ It follows that $\lambda_l = \cos{L} \sin{\psi}.$ Differentiating Eq.~\eqref{eq: lambda_psi_id_zero_equation} and simplifying using Eqs.~\eqref{eq: Sabban_frame_equations_rate_change_heading} and \eqref{eq: evolution_lambda_psi_dot},
    \begin{align*}
        \left(\lambda_L \cos{\psi} + \lambda_l \frac{\sin{\psi}}{\cos{L}} \right) \dot{\psi} - \lambda_l \frac{\sin{L}}{\cos^2{L}} \equiv 0.
    \end{align*}
    Noting that $\lambda_\psi \equiv 0,$ the Hamiltonian reduces to $H = e + \lambda_L \cos{\psi} + \lambda_l \frac{\sin{\psi}}{\cos{L}}$, which is identically zero. Hence, from the above equation, it follows that for $e = -1,$
    $\dot{\psi} = \lambda_l \frac{\sin{L}}{\cos^2{L}}.$
    Using the expression for $\dot{\psi}$ from Eq.~\eqref{eq: Sabban_frame_equations_rate_change_heading}, it follows that
    \begin{align} \label{eq: solving_for_u_lambda_psi_zero}
        \lambda_l \frac{\sin{L}}{\cos^2{L}} = \tan{L} \sin{\psi} + \frac{1}{R} u.
    \end{align}
    Substituting the solution $\lambda_l = \cos{L} \sin{\psi}$ in the above equation, it follows that $\frac{1}{R} u \equiv 0 \implies u \equiv 0.$
\end{proof}

Now, to show that the model with the spherical coordinates is equivalent to the model with the Sabban frame, it is desired to show that
\begin{itemize}
    \item The evolution of the adjoint variable corresponding to the control input in the two models, which are $\lambda_\psi$ and $H_{12}$, are the same.
    \item The optimal control inputs in the two models are the same. 
\end{itemize}

\begin{lemma}
    The evolution of $\lambda_\psi$ in the model with spherical coordinates is the same as that of the $H_{12}$ in the model with the Sabban frame if $U_{max} = \frac{1}{\eta}$.
\end{lemma}
\begin{proof}
    Consider the expression for the derivative of $\lambda_\psi$ given in Eq.~\eqref{eq: evolution_lambda_psi_dot}. Noting that $\dot{\lambda}_\psi$ is differentiable, the expression for its second derivative can be obtained by differentiating Eq.~\eqref{eq: evolution_lambda_psi_dot} and simplifying as
    \begin{align*}
        \Ddot{\lambda}_{\psi} &= 
        -\lambda_\psi \left(1 + \frac{u^2}{\eta^2} \right) - e \frac{u}{\eta},
    \end{align*}
    where $e \in \{-1, 0\}$ and $-1 \leq u \leq 1$. It should be recalled that the evolution of $H_{12}$ is given by the differential equation in Eq.~\eqref{eq: evolution_H12}, where $\kappa (s)$ is defined in Eq.~\eqref{eq: optimal_curvature} and lies in $[-U_{max}, U_{max}]$. Hence, it can be concluded the evolution of $H_{12}$ is the same as $\lambda_\psi$ if $U_{max} = \frac{1}{\eta}.$
\end{proof}


Now, it remains to be shown that the optimal segments in the alternate model considered correspond to tight circular arcs of radius $\frac{\eta}{\sqrt{1 + \eta^2}}$, and correspondingly, the bound $U_{max}$ in the model with the Sabban frame relates to $\eta$ in the alternate model through $U_{max} = \frac{1}{\eta}$.

\begin{lemma}
    The control input $u = \pm 1$ corresponds to the vehicle moving a circular arc on the sphere with radius $r = \frac{\eta}{\sqrt{1 + \eta^2}}$ and $u = 0$ corresponds to the vehicle moving on a great circular arc.
\end{lemma}
\begin{proof}
    To show that the vehicle travels on a circular arc, it suffices to show that the vehicle travels on a plane intersected with the considered unit sphere. For this purpose, the initial conditions on $L, l,$ and $\psi$ can be chosen without loss of generality by utilizing a coordinate transformation. Consider the equation for $\dot{\psi}$ given in Eq.~\eqref{eq: Sabban_frame_equations_rate_change_heading}, which can be observed to depend on the initial condition chosen for $L$ and $\psi$. It is desired to choose an initial condition such that $\dot{\psi} \equiv 0.$
    Suppose $\psi (0) = \frac{\pi}{2}.$ Then, $L (0)$ is selected such that $\dot{\psi} \equiv 0,$ i.e.,
    \begin{align*}
        \tan{L (0)} = -\frac{1}{\eta} u \quad \implies L (0) = \tan^{-1} \left(-\frac{1}{\eta} u \right).
    \end{align*}
    It should be noted that such a choice of $\psi (0)$ implies that $L$ is constant, since $\dot{L} = \cos{\psi}$
    and $\cos{\psi (0)} = 0.$ Therefore, $\psi$ and $L$ are constants over the considered segment. Choosing $l (0) = 0,$ the plane containing the segments corresponding to $u = 0, 1,$ and $-1$ are shown in Figs.~\ref{subfig: great_circle}, \ref{subfig: tight_turn_1}, and \ref{subfig: tight_turn_2}, respectively. It can be observed that the radius of the great circular segment equals $1,$ whereas the tight turns (corresponding to $u \pm 1$) equal $\frac{\eta}{\sqrt{1 + \eta^2}}$.
    \begin{figure}[htb!]
         \centering
         \begin{subfigure}[b]{0.33\textwidth}
             \centering
             \includegraphics[width=\textwidth]{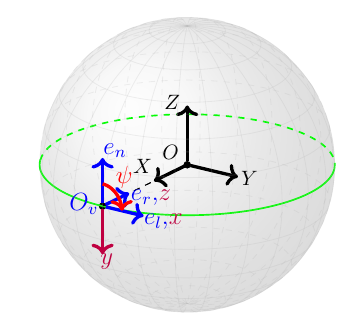}
             \caption{Great circular corresponding to $u = 0$}
             \label{subfig: great_circle}
         \end{subfigure}
         \hfill
         \begin{subfigure}[b]{0.33\textwidth}
             \centering
             \includegraphics[width=\textwidth]{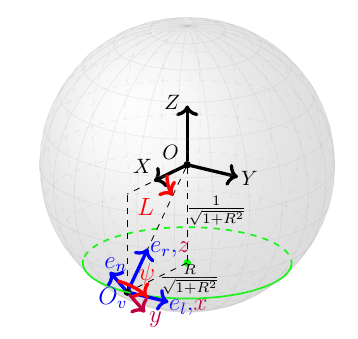}
             \caption{Tight turn corresponding to $u = 1$}
             \label{subfig: tight_turn_1}
         \end{subfigure}
         \hfill
         \begin{subfigure}[b]{0.33\textwidth}
             \centering
             \includegraphics[width=\textwidth]{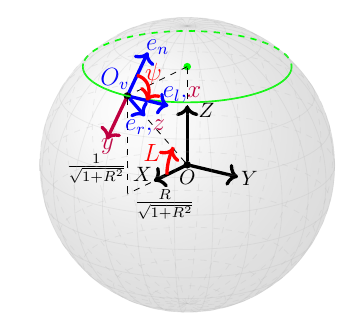}
             \caption{Tight turn corresponding to $u = -1$}
             \label{subfig: tight_turn_2}
         \end{subfigure}
        \caption{Optimal segments for alternate sphere model}
        \label{fig: optimal_segments_alternate_sphere_model}
    \end{figure}
    It remains to be shown that $u = 1$ corresponds to a left turn, whereas $u = -1$ corresponds to a right turn. For this purpose, consider the initial condition for $L, l,$ and $\psi$ to be $0, 0, \frac{\pi}{2}$. In this case, consider $\dot{\psi}$ given in Eq.~\eqref{eq: Sabban_frame_equations_rate_change_heading}, which reduces to $\dot{\psi} = \frac{1}{\eta} u$.
    For $u = 1,$ $\dot{\psi}$ is positive, which implies that the vehicle takes a right turn (refer to Fig.~\ref{fig: frames_sphere}), whereas for $u = -1,$ $\dot{\psi}$ is negative, which implies that the vehicle takes a left turn. Therefore, if $\lambda_\psi > 0,$ the vehicle takes a right turn, whereas if $\lambda_\psi < 0,$ the vehicle takes a left turn. It must be noted here that if $H_{12} > 0$ in the Sabban frame model, $u_g = -U_{max},$ which corresponds to a right turn (refer to Eq.~\eqref{eq: optimal_control_inputs}), whereas if $H_{12} < 0,$ the vehicle takes a left turn. In this regard, the sign of the adjoint variable in both models yields the same type of turn for the vehicle.
\end{proof}

Having shown that the evolution of the adjoint variables in the two models are the same, the optimal control actions are the same, and the sign of the two adjoint variables dictate the same type of turn being taken in both the models, it follows that the two models are equivalent. In this regard, the following theorem is immediate for the motion planning problem on a sphere using spherical coordinates using the results from \cite{monroy, 3D_Dubins_sphere}.

\begin{theorem}
    For the model employing spherical coordinates for motion planning on a sphere, the optimal path is of type $CGC, CCC,$ or a degenerate path of the same for $r = \frac{\eta}{\sqrt{1 + \eta^2}} \in (0, \frac{1}{2}] \bigcup \{\frac{1}{\sqrt{2}} \}$.
\end{theorem}

\section{Conclusion}

In this article, two models were studied for motion planning of a Dubins vehicle on a sphere. In the first model, the vehicle's location and orientation were described using a frame, whereas in the second model, spherical coordinates and a heading angle were used. While the first model has been utilized in the literature to obtain the optimal path type for $r \in (0, \frac{1}{2}] \bigcup \{\frac{1}{\sqrt{2}} \}$, which is the turning radius of the tightest turn, the latter model has not been derived in the literature. Hence, the model using the latter method is derived and formulated as an optimal control problem in this article, which is applicable to the aerospace community. Furthermore, the equivalence of the two models is shown by proving that the optimal control actions are the same and the adjoint variables controlling the control action evolve through the same evolution equation. By showing the equivalence of these two models, the motion planning problem using the spherical coordinate description is also solved, wherein the optimal path is obtained to be of type $CGC, CCC,$ or a degenerate path of the same for $r \in (0, \frac{1}{2}] \bigcup \{\frac{1}{\sqrt{2}} \}$.

\bibliography{sample}

\end{document}